\documentclass[10pt]{amsproc}
\newtheorem{theorem}{\sc Theorem}[section]
\newtheorem{lemma}[theorem]{\sc Lemma}
\newtheorem{proposition}[theorem]{\sc Proposition}

\newcommand{\U}{\mathcal{U}}
\newcommand{\F}{\Delta (G)} 

\newcommand{\y}{z}
\newcommand{\ol}{\overline}
\usepackage{MnSymbol}

\begin{document}

\title[Centralizers of commutators]{Commutators,  centralizers, and strong conciseness in profinite groups}
\thanks{The first and second authors are members of GNSAGA (INDAM), 
and the third author was  supported by  FAPDF and CNPq.}

\author{Eloisa Detomi}
\address{Dipartimento di Ingegneria dell'Informazione - DEI, Universit\`a di Padova, Via G. Gradenigo 6/B, 35121 Padova, Italy} 
\email{eloisa.detomi@unipd.it}
\author{Marta Morigi}
\address{Dipartimento di Matematica, Universit\`a di Bologna\\
Piazza di Porta San Donato 5 \\ 40126 Bologna \\ Italy}
\email{marta.morigi@unibo.it}
\author{Pavel Shumyatsky}
\address{Department of Mathematics, University of Brasilia\\
Brasilia-DF \\ 70910-900 Brazil}
\email{pavel@unb.br}


\begin{abstract} 
{A group $G$ is said to have restricted centralizers if for each $g \in G$ the centralizer $C_G(g)$ either is finite or has finite index in $G$. Shalev showed that a profinite group with restricted centralizers is virtually abelian. We take interest in profinite groups with restricted centralizers of uniform  commutators, that is, elements of the form $[x_1,\dots,x_k]$, where $\pi(x_1)=\pi(x_2)=\dots=\pi(x_k)$. Here $\pi(x)$ denotes the set of prime divisors of the order of $x\in G$. It is shown that such a group necessarily has an open nilpotent subgroup. We use this result to deduce that $\gamma_k(G)$ is finite if and only if the cardinality of the set of uniform $k$-step commutators in $G$ is less than $2^{\aleph_0}$.}
\end{abstract}
\maketitle

\section{Introduction} 
A group $G$ is said to have restricted centralizers if for each $g \in G$ the centralizer $C_G(g)$ either is finite or has finite index in $G$. This notion was introduced by Shalev in \cite{shalev} where he showed that a profinite group with restricted centralizers is virtually abelian. As usual, we say that a profinite group has a property virtually if it has an open subgroup with that property. Recently, profinite groups with restricted centralizers of some specific elements were considered in \cite{as, DMSrestricted}.
The article \cite{DMSrestricted} handles profinite groups with restricted centralizers of $w$-values for a multilinear commutator word $w$. The theorem proved in \cite{DMSrestricted} says that if $w$ is a multilinear commutator word and $G$ is a profinite group in which the centralizer of any $w$-value is either finite or open, then the verbal subgroup $w(G)$ generated by all $w$-values is virtually abelian. 
Recall that the  lower central words $\gamma_k$ are  recursively defined by $$\gamma_1=x_1,\qquad \gamma_k=[\gamma_{k-1}(x_1,\ldots,x_{k-1}),x_k].$$ Of course  $\gamma_k(G)$ is the $k$-th  term of the lower central series of $G$. Thus, if the $\gamma_k$-values have restricted centralizers in a profinite group $G$, then $\gamma_k(G)$ is virtually abelian.

 In this paper we will show that if the $\gamma_k$-values have restricted centralizers, then the group $G$ is virtually nilpotent. In fact, we will establish a stronger result.

If $G$ is a profinite group, then $|G|$ denotes its order, which is a  Steinitz number, and $\pi(G)$ denotes the set of prime divisors of $|G|$. Similarly, if $g$ is an element of $G$, then $|g|$ and $\pi(g)$ respectively  denote the order of the procyclic subgroup generated by $g$ and the set of prime divisors of $|g|$.

We will say that an element $g$ of a profinite group $G$ is a {\it uniform $k$-step commutator } (u$_k$-commutator for short) if there are elements $x_1,x_2,\ldots,x_k\in G$ such that $\pi(x_1)=\dots=\pi(x_k)$ and $g=[x_1,x_2,\ldots,x_k]$.
 When $k=2$,  the element $g$ will be referred to simply as a {\it uniform commutator} (such elements were called anti-coprime commutators in \cite{dms2,dms}).

\begin{theorem}\label{main} Let  $G$ be a profinite group in which the centralizers of uniform $k$-step commutators  are either finite or open. Then $G$ is virtually nilpotent and $\gamma_k(G)$ is virtually abelian.
\end{theorem}

We do not know if  there exists a constant, say $C$, depending only on $k$ such that any group $G$
  satisfying the hypothesis of Theorem \ref{main}   has an open nilpotent subgroup of class at most $C$. In the case where $k=2$ the affirmative answer is furnished by the next result. A key tool employed in the proof was established in \cite{ES} using probabilistic arguments.

\begin{theorem}\label{main2} Let  $G$ be a profinite group in which the centralizers of uniform commutators are either finite or open. Then $G$ has an open  subgroup which is  nilpotent of class at most $3$. 
\end{theorem}

A somewhat unexpected by-product of the proof of Theorem \ref{main} is related to the concept of strong conciseness in profinite groups introduced in \cite{dks}. A group-word $w$ is strongly concise if  the verbal subgroup $w(G)$ is finite in any profinite group $G$  in which $w$ takes less than   $2^{\aleph_0}$ values. A number of     recent results on strong conciseness of group-words can be found in \cite{dks, joao, KS-strong}. The concept of strong conciseness can be applied in a wider context:  suppose $\varphi(G)$ is a  subset that can be naturally defined in every profinite group $G$, then one can ask whether the subgroup generated by $\varphi(G)$ is finite whenever  $|\varphi(G)| < 2^{\aleph_0}$. 

It was shown in \cite{dms} that this is the case if $\varphi(G)$ is the set of u$_2$-commutators; that is, a profinite group $G$ has finite commutator subgroup $G'$ if and only if the cardinality of the set of u$_2$-commutators in $G$ is less than $2^{\aleph_0}$. In view of this, it was natural to ask whether finite-by-nilpotent profinite groups  
admit a similar characterisation. Now we are able to answer the question in the affirmative.

\begin{theorem}\label{strong} Let $k\geq1$, and let $G$ be a profinite group. Then $\gamma_k(G)$ is finite if and only if the cardinality of the set of uniform $k$-step commutators  in $G$ is less than $2^{\aleph_0}$. 
\end{theorem}

Note that   in Theorem \ref{strong} 
 the order of $\gamma_k(G)$ is bounded in terms of the number of u$_k$-commutators (see Proposition \ref{boundedly}). 
It would be interesting to see if also profinite groups $G$ in which the $k$-th term  of the derived series is finite can be characterised in the same spirit. For now this remains an open problem.

\section{Preliminaries}

Results on finite groups often admit a
natural interpretation for profinite groups (see, for example, \cite{rib-zal} or \cite{wilson}). Throughout the paper we use certain profinite versions of facts on finite groups without explaining in detail how these can be deduced from the corresponding finite cases. On all such occasions the deduction can be performed via the routine inverse limit argument. Every homomorphism of profinite groups considered in this paper is continuous, and every subgroup of a profinite group is closed, unless otherwise specified.

If $x$ is an element of a group $G$, we write $x^G$ for the conjugacy class of $x$ in $G$.  On the other hand, if $K$ is a subgroup of $G$, then $K^G$ denotes the normal closure of $K$ in $G$, that is, the subgroup generated by all conjugates of $K$ in $G$, with the usual convention that if $G$ is a topological group then $K^G$ is a closed subgroup.

We will denote by $\F $  the set of $FC$-elements of $G$, i.e. 
$$\F =\{ x\in G \mid |x^G| < \infty\}.$$ 
Obviously  $\F$ is a normal abstract subgroup of $G$. Note that if $G$ is a profinite group, $\F$ does not need to be closed. 

 Recall that a pro-$p$ group is an inverse limit of finite $p$-groups, a pro-$\pi$ group is an inverse limit of finite $\pi$-groups, a pronilpotent group is an inverse limit of finite nilpotent groups, a prosoluble group is an inverse limit of finite soluble groups. 
 
It is well-known that a finite-by-prosoluble group is virtually prosoluble,  and a  profinite group $G$ that is an extension of a prosoluble group $N$ by
a prosoluble group $G/N$ is prosoluble (see e.g. \cite[Lemma 2.2]{KS}).

If $G$ is a profinite group and $\pi$ a set of primes,  $O_{\pi} (G)$ stands for  the maximal normal pro-$\pi$ subgroup of $G$. 

Recall that $\gamma_\infty(G)$ stands for the intersection of the terms of the lower central series of a group $G$. A profinite group $G$ is pronilpotent if and only if $\gamma_\infty(G)=1$. Each profinite group $G$ has a maximal pronilpotent normal subgroup, its Fitting subgroup $F(G)$. Set $F_0(G)=1$ and $F_{i+1}(G)/F_i(G)=F(G/F_i(G))$ for every integer $i\ge 0$. We say that the profinite group $G$ has (finite) Fitting height $h=h(G)$ if $G=F_h(G)$ and $h$ is the least integer with this property. Obviously, $G$ has finite Fitting height at most $h$ if, and only if, $G$ is an inverse limit of finite soluble groups of Fitting height at most $h$. A profinite group $G$
 is metapronilpotent if and only if it has Fitting height at most $2$ or, equivalently, if and only if $\gamma_\infty (G)$ is pronilpotent. As usual, $Z(G)$ denotes the centre of the group $G$.

We mention the following result about nilpotent profinite groups. 
\begin{lemma}\label{center} 
Let  $G$ be an infinite nilpotent profinite group. Then $Z(G)$ is infinite. 
\end{lemma} 
\begin{proof} Assume that  $Z(G)$ is finite. Then there exists an open normal subgroup $K$ of $G$ such that $K \cap Z(G)=1$. Since $G$ is nilpotent, this implies that $K=1$ and so $G$ is finite, a contradiction.
\end{proof}

Profinite groups have Sylow $p$-subgroups and satisfy analogues of the Sylow theorems. 
Any prosoluble group $G$ has a Sylow basis (a family of pairwise permutable
 Sylow $p_i$-subgroups $P_i$ of $G$, exactly one for each prime $p_i \in \pi(G)$), and any two Sylow bases are conjugate (see \cite[Proposition 2.3.9]{rib-zal}). The basis normalizer (also known as the system normalizer) of such a Sylow basis in $G$ is $T =\bigcap_{i} N_G(P_i).$ If $G$ is a prosoluble group and $T$ is a basis normalizer in $G,$ then $T$ is pronilpotent and  $G = \gamma_\infty (G) T$ (see \cite[Lemma 5.6]{reid}). 

The set of uniform $k$-step commutators of $G$ will be denoted by $\U_k(G)$  and we write $\U_k=\U_k(G)$ when it is clear which group we are referring  to. Remark that $\U_k$ is symmetric, that is, it is closed under taking inverses, since $[x,y]^{-1}=[x^y, y^{-1}]$. Moreover, $\U_k(G/N)=\U_k(G)N/N$ whenever $N$ is a normal subgroup of $G$. 

If $x,y$ are elements of a group $G$ and $k$ is a positive integer, $[x,{}_k y]$ is recursively defined by $[x,{}_1y]=[x,y]$ and $[x,{}_{k+1}y]=
[[x,{}_k y],y]$ for every $k\ge 1$. We will often use the fact that if $G$ is profinite, then the element $[x,y,y]=[y^{-xy},y]$ is a uniform commutator.
Note that if the profinite group $G$  is pronilpotent, then it is a direct product of its Sylow subgroups and so every $\gamma_k$-value in $G$ is a u$_k$-commutator.

\begin{lemma}\label{gen1}
 Let $G$ be a profinite group.  Then the set $\U_k$ generates $\gamma_k(G)$.
\end{lemma}
\begin{proof} If $k=1$, then $\U_k=G$ and we have nothing to prove. Therefore we assume that $k\geq2$. Let $N$ be the subgroup generated by $\U_k$. Obviously, $N\le \gamma_k(G)$ so we only need to show that $\gamma_k(G)\leq N$. Recall that $[x,y,y]=[y^{-xy},y]$ is a uniform commutator for any $x,y\in G$. It follows that if $\bar x=Nx$ and $\bar y=Ny$ are elements of $G/N$, then $[\bar x,_{k} \bar y]=1$, as $[y^{-xy},{}_{k-1}y]\in \U_k$. Since finite Engel groups are nilpotent (see \cite[12.3.4]{rob}), we deduce that $G/N$ is pronilpotent. We mentioned earlier that in a pronilpotent group every $\gamma_k$-value is a u$_k$-commutator. Therefore $G/N$  
 is nilpotent of class at most $k-1$, whence $\gamma_k(G)\le N$.
\end{proof}

If $\phi$  is an automorphism of a finite group $H$ of coprime order, that is, such
that $(|\phi |, |H|) = 1$, then we say for brevity that $\phi$  is a coprime automorphism
of $H.$ This definition is extended to profinite groups as follows. We say that $\phi$  is a
coprime automorphism of a profinite group $H$ meaning that the procyclic group $\langle\phi \rangle$ 
faithfully acts on $H$ by continuous automorphisms and $\pi(\langle\phi \rangle ) \cap \pi(H)=\emptyset$. Since the
semidirect product $H \langle\phi \rangle$ is also a profinite group, $\phi$ is a coprime automorphism of $H$
if and only if for every open $\phi$-invariant normal subgroup $N$ of $H$ the automorphism
(of finite order) induced by $\phi$ on $H/N$ is a coprime automorphism. 
 We will need some profinite equivalent of well-known results about coprime actions in finite groups (see e.g. \cite{KS}). 

\begin{lemma} \label{bul-invariant}
If $\phi$ is a coprime automorphism of a profinite group $G,$ then for every
prime $q \in \pi(G)$ there is a $\phi$-invariant Sylow $q$-subgroup of $G.$ 
\end{lemma}

\begin{lemma}\cite[Lemma 4.6]{KS} \label{bul} Let $\phi$ be a coprime automorphism of a finite nilpotent group $G$. Then the set of elements of the form $[g,\phi]$, where $g\in G$, coincides with the set of  elements of the form $[g,\phi,\phi]$.
\end{lemma}

A repeated application of the previous lemma yields the following result for profinite groups: 
\begin{lemma}\label{bul2} Let $\phi$ be a coprime automorphism of a pronilpotent group $G$ and let $k$ be a positive integer. Then  the set of  elements of the form $[g,\phi]$, where $g\in G$, coincides with the set of elements of the form $[g,_{k}\phi]$.
\end{lemma}
\begin{proof}
For every positive integer $n$, let $\alpha_n$ be the continuous map defined on $G$ by $x \mapsto [x,_n \phi]$. We can repeatedly   apply Lemma  \ref{bul} to every finite image $\bar G$  of $G$, which is  nilpotent,  to get 
$\alpha_1(\bar G)=\alpha_2(\bar G)$, and so also  $\alpha_1(\bar G)=\alpha_k(\bar G)$. Since this holds for every finite image of $G$ and the maps are continuous, we deduce that $\alpha_1( G)=\alpha_k( G)$. 
\end{proof}

\section{On groups in which centralizers of word values are either finite or open}  

In this section we handle the important particular case of Theorem \ref{main} where the centralizers of all $\gamma_k$-values in $G$ are either finite or open. We will use some results from \cite{DMSrestricted}, which were proved for general multilinear commutator words.

Multilinear commutator words are words which are obtained by nesting commutators, but using always different variables. More formally, the word $w(x) = x$ in one variable is a multilinear commutator; if $u$ and $v$ are  multilinear commutators on disjoint sets of variables then the word $w=[u,v]$ is a multilinear commutator, and all multilinear commutators are obtained in this way.

Clearly, the lower central words $\gamma_k$ are particular instances of multilinear commutator words.

Throughout the article, if $w$ is a word and $G$ is a group, $G_w$ will denote the set of all $w$-values in $G$. 
  
We will need a combinatorial lemma proved in \cite{DMSrestricted}.

\begin{lemma}\cite[Lemma 3.4]{DMSrestricted}\label{comb1}
 Let $w=w(x_1, \dots, x_n)$ be a multilinear commutator word. 
Assume that  $T$ is a normal subgroup of a group $G$ and 
 $a_1, \dots , a_n$ are elements of $G$ such that every element in the set 
$\{w(a_1t_1,\dots,a_nt_n) \mid t_1,\dots,t_n\in T\}$ has at most $m$ conjugates in $G$.
Then every element in $T_w$ has at most $m^{2^n}$ conjugates in $G$.
\end{lemma}

We will first consider profinite groups where  every $\gamma_k$-value is an $FC$-element. The following result was recently obtained in \cite{S-BFC}. Here $(k,m)$-bounded means bounded by a function of the parameters $k$ and $m$.

\begin{theorem}\label{dadada} 
Let $k\geq1$ and $G$ be a group in which $|x^G|\leq m$ for any $\gamma_k$-value $x\in G$. Then $G$ has a nilpotent subgroup of $(k,m)$-bounded index and $(k,m)$-bounded class. 
\end{theorem}

To deal with the case $k=2$ we require a more complicated result from \cite{ES} (see Theorem 1.2  and the  preceding comments).

\begin{theorem}\label{ESgamma} If $G$ is a group in which $|x^G|\leq m$ for any commutator $x\in G$, then
$G$ has a subgroup $H$ of nilpotency class at most $4$ such that $[G : H]$ and $|\gamma_4(H)|$ are both finite and $m$-bounded.
\end{theorem}

\begin{proposition}\label{profinite-FC} 
Let $k$ be a positive integer and $G$ a profinite group 
 in which every $\gamma_k$-value is an $FC$-element.  Then $G$ is virtually nilpotent. If $k=2$ then $G$  has
an open  subgroup $H$ of nilpotency class at most $3$. 
\end{proposition} 
\begin{proof}
For each positive integer $j$ consider the set $\Delta_j$ of elements $g\in G$ such that $|g^G|\le j$.  Note that the sets $\Delta_j$ are closed (see for instance   \cite[Lemma 5]{LP}). Set: 
$$C_j=\{(g_1,\dots,g_k) \mid  g_i\in G {\textrm{ and }} [g_1,\dots,g_k] \in \Delta_j\}.$$

Each set $C_j$ is closed in $G \times \cdots \times G$, being  the inverse image of the closed set  $\Delta_j$ under the continuous map $(g_1,\dots, g_k)\mapsto [g_1, \dots , g_k]$. 
Moreover the union of the sets $C_j$ is the whole group $G \times \cdots \times G$. By the Baire category theorem (cf. \cite[p.\ 200]{Ke}) at least one of the sets $C_j$ has nonempty interior. Hence, there exists a positive integer $m$, some elements $\y_i\in G$, and an open normal subgroup $T$ of $G$ such that 
$$[\y_1 T,\dots,\y_nT]\subseteq \Delta_{m}.$$ 

By Lemma \ref{comb1}, every  $\gamma_k$-value in elements of $T$ has at most $m^{2^k}$ conjugates in $G$. It follows from Theorem \ref{dadada} that 
 $T$ has a nilpotent abstract subgroup $B$ of $(k,m)$-bounded index and $(k,m)$-bounded class. 
Then the topological closure of $B$ is an open nilpotent subgroup of $G$. Thus $G$ is virtually nilpotent, as claimed.

If $k=2$, by Theorem \ref{ESgamma} the subgroup $T$ contains a nilpotent abstract subgroup $B$ of $m$-bounded index and such that $\gamma_4(B)$ has $m$-bounded order. Then the topological closure $\ol B$ of $B$ is an open subgroup of $G$ such that $\gamma_4(\ol B)$ has finite order. Choose an open subgroup $K\le G$ such that $K\cap \gamma_4(\ol B)=1$. Observe that $K\cap \ol B$ is open in $G$ and nilpotent of class at most $3$, as required.
\end{proof}

We will require the following corollary of the main result in \cite{DMSrestricted}. 
\begin{lemma}\label{openT} \cite[Corollary 1.2]{DMSrestricted}
Let $w$ be a multilinear commutator word and $G$ a profinite group in which the centralizers of $w$-values 
 are either finite or open. Then  $G$ has an open subgroup $T$ such that $w(T)$ is abelian.  
\end{lemma}

\begin{proposition}\label{maingamma}
Let $k$ be a positive integer and $G$ a profinite group in which the centralizers of $\gamma_k$-values 
 are either finite or open. Then $G$ is virtually nilpotent. If $k=2$, then $G$ has an open  subgroup which is  nilpotent of class at most $3$. 
 \end{proposition}
 \begin{proof}  By Lemma \ref{openT}, $G$ has an open subgroup $T$ such that $\gamma_k(T)$ is abelian. Without loss of generality, we can assume that $G=T$. 
 
 If all $\gamma_k$-values of $G$ are $FC$-elements, then by Proposition \ref{profinite-FC} we conclude that  $G$ is virtually nilpotent. In particular,  if  $k=2$, then $G$ has an open  subgroup which is  nilpotent of class at most $3$. 
  
So assume that  there exists a $\gamma_k$-value whose centralizer  in $G$ is finite.  As $\gamma_k(G)$ is abelian,  we conclude that $\gamma_k(G)$ is finite. 
 Therefore there exists an open normal subgroup $N$ of $G$ such that $N \cap \gamma_k(G)=1$. In particular, $\gamma_k(N)=1$ and $N$ is an open nilpotent subgroup of $G$. If $k=2$ then $N$ is abelian and $G$ is virtually abelian. This concludes the proof.
  \end{proof}

\section{On groups in which centralizers of uniform commutators are finite} 

In this section we will prove 
 Theorem \ref{main} in the special case where the elements of $\U_k$ have finite centralizers (see Proposition \ref{mainFinite}).
Throughout, it will be assumed that $k\geq2$. Note that if $G$ is a profinite group where all $\gamma_k$-values have finite centralizers, then $G$ is either finite or nilpotent of class at most $k-1$ \cite[Corollary 1.3]{DMSrestricted}.

We will repeatedly use the following observation: 

\begin{lemma}\label{finite} Let  $G$ be a profinite group in which the centralizers of {\rm u}$_k$-commutators are finite.
If  $H$ is a pronilpotent subgroup of $G$, then either $H$ is finite or $H \cap  \U_k =1$. In the latter case,  $H$ is nilpotent of class at most $k-1$. 
\end{lemma}
\begin{proof}
Since $H$ is a direct product of its Sylow subgroups, it follows that any $\gamma_k$-value of $H$ lies in $\U_k$ and, by  \cite[Corollary 1.3]{DMSrestricted},  $H$ is either finite or  nilpotent of class at most $k-1$. Assume  $H \cap  \U_k \neq 1$ and let $y$ be a nontrivial element of $H \cap  \U_k $.  Since $Z(H) \le C_G(y)$, it follows that $Z(H)$ is finite. In view of  Lemma \ref{center} this implies that $H$ is finite. 
\end{proof} 

\begin{proposition}\label{mainFinite} Let  $G$ be a profinite group in which the centralizers of {\rm u}$_k$-commutators are finite. Then $G$ is either finite or  nilpotent of class at most $k-1$. 
\end{proposition}
\begin{proof} If $G$ is pronilpotent, then the result is immediate from the previous lemma. Assume that $G$ is not pronilpotent. We want to prove that $G$ is finite. We will use the fact that the Sylow subgroups of $G$ are either finite or  nilpotent of class at most $k-1$. 

The profinite version of Burnside's theorem  \cite[Theorem 3.3]{gilotti-ribes-serena} says that if $N_G(P)/C_G(P)$ is a pro-$p$ group for every nontrivial pro-$p$ subgroup $P$ of $G$, then $G$ has a normal $p$-complement. Note that a group is pronilpotent whenever it has normal $p$-complement for every prime $p$. As our group $G$ is not pronilpotent, there exists a prime $p$ and a pro-$p$ subgroup $P\leq G$ such that $N_G(P)/C_G(P)$ contains a nontrivial  $p'$-element. So $N_G(P)\setminus C_G(P)$ contains a  $p'$-element $a$ that induces a nontrivial coprime automorphism on $P$. We deduce from Lemma \ref{bul2} that there exists a nontrivial $p$-element  $y=[x,_k a]\in\U_k \cap P$. It follows from Lemma \ref{finite} that the Sylow $p$-subgroups of $G$ are finite. 

Let $N$ be an open normal pro-$p'$ subgroup of $G$ intersecting $C_G(y)$ trivially. Then $C_N(y)=1$ and $y$ acts coprimely on $N$.
Let $q$ be a prime in $\pi(N )$.  By Lemma \ref{bul-invariant},  there is a $y$-invariant Sylow $q$-subgroup $Q$ of $N$.  
   
 As $C_N(y)=1$, the map $x\mapsto [x,y]$  is injective on $Q$. Hence, by  Lemma \ref{bul2},  $\U_k \cap Q \neq 1$. Thus $Q$ is a finite $q$-group by Lemma \ref{finite},  and the map  $x \mapsto [x,y]$ is also surjective. 
 Therefore  $Q\subseteq \U_k$  by  Lemma \ref{bul2}. 
 
 Choose an element $z\in Q$ of prime order $q$.
Since $C_G(z)$ is finite, there exists an open normal $q'$-subgroup $K$ of $G$ that has trivial intersection with $C_G(z)$. In particular, $z$ acts coprimely and fixed-point-freely on $K$. As $z$ has prime order, combining  the well-known results of Thompson and Higman \cite{hi,tho} (see also \cite[Theorem 2.6.2]{wilson}) we conclude that $K$ is nilpotent. Again, by Lemma \ref{bul2}, we deduce that $K\subseteq \U_k$, whence $K$ is finite by Lemma \ref{finite}. It follows that $G$ is finite. 
\end{proof}

\section{ On groups in which uniform commutators are FC-elements}  

In this section we will prove the following proposition. 

\begin{proposition}\label{mainFC} Let  $G$ be a profinite group such that 
  $\U_k \subseteq \Delta(G)$. Then $G$ is virtually nilpotent. 
\end{proposition}

 We will need a technical  results  about commutators.

\begin{lemma}\label{lem2} 
Let $x_1, \dots , x_n $ be elements of a group $G$ and let $y\in\{x_1, \dots , x_n\}$. Then $[x_1, \dots, x_n]$ is a product of $2^{n-1}$ conjugates of $y^{\pm 1}$. 
\end{lemma}
\begin{proof} 
Using basic commutator identities it can be easily seen  that  if $v$ is a product of $t$ conjugates of $y^{\pm 1}$,  then both
 $[v, x]= v^{-1} v^x$ and $[x,v] =v^{-x} v$ are  products  of $2t$ conjugates of $y^{\pm 1}$. Then, by induction on $n$, it follows that $[x_1, \dots, x_n]$ is a product of  $2^{n-1} $ conjugates of $y^{\pm 1}$. 
\end{proof}

Suppose that a set $\Sigma$ is a union of finitely many subsets $\Sigma_1,\dots,\Sigma_m$. Then $\Sigma$ admits a partition with blocks $\Omega_1,\dots,\Omega_s$ such that $s \le 2^m$ and each $\Sigma_i $ is a  (disjoint) union of some of the blocks $\Omega_1,  \dots,\Omega_s$. Formally, we can take $\Omega_1,\dots,\Omega_s$ to be all the nonempty sets of the form 
  \[ 
 \Gamma_1\cap \Gamma_2\cap \dots\cap \Gamma_m, \quad \textrm{ where } 
\Gamma_i  \textrm{ is either } \Sigma_i  \textrm{ or } \Sigma\setminus\Sigma_i   \textrm{ for all }  i.\]

This will be used in the next lemma.
 
\begin{lemma}\label{idea} Let $N$ be a normal pronilpotent subgroup of a profinite group G. Let $X\subseteq G$ be the set of commutators $[g_1,...,g_k]$ where at least one of the entries $g_1,...,g_k$ belongs to $N$. Then every element of $X$ is a product of $k$-boundedly many elements in $\U_k\cap N$.
\end{lemma}

\begin{proof}  
 Let us fix an element $ [g_1,\dots,g_k]\in X$, with $g_i \in N$. 
  The set 
  \[\pi (g_1) \cup \dots \cup \pi(g_k) \] 
  admits a finite partition with blocks $\pi_1,\dots, \pi_s$ with the property that each $\pi(g_t)$ is a union of some of these blocks. Note that $s\le 2^k$. 
  So each element $g_t$ is a product  
  \[ g_t=\prod_{j=1}^s y_{tj} \textrm{ where }y_{tj}\in\langle g_t\rangle,\,\, \pi(y_{tj})=\pi_j \textrm{ whenever } y_{tj}\ne 1. \]
    Repeatedly using the commutator  identity $[ab,c]=[a^b,c^b]\,[b,c]$ we obtain that $[g_1,\dots,g_k]$  is a product of $k$-boundedly many  elements of the form: $y=[u_1,\dots ,u_k]$, where, for each $t$, $u_t$ is a conjugate of some $y_{tj}$. In particular $u_i\in N$, and for each $t,j=1,\dots, k$ the sets  $\pi(u_t)$ and $\pi(u_j)$ are either equal or disjoint. 
  So, taking into account that $N$ is normal, we just need to show that every such  element $y= [u_1,\dots,u_k]$ is a product of $k$-boundedly many elements in $\U_k \cap N$.
	
Let $\tilde \pi=\pi(u_i)$. As $N$ is pronilpotent and normal in $G$, the subgroup  $O_{\tilde \pi}(N)$ is normal in $G$. 	In the sequel, we will  use the following fact:  If $a\in  O_{\tilde \pi}(N)$
 and $b$ induces a coprime automorphism on $O_{\tilde \pi}(N)$, then $[a,b]\in \U_k$. Indeed, by Lemma \ref{bul2}, there exists $d\in O_{\tilde \pi} (N)$ such that $[a, b]=[d,{}_{k}b] = [b^{-db},{}_{k-1}b]$.
	
	If $\tilde \pi=\pi(u_1)=\dots=\pi(u_k)$ then $y\in \U_k \cap N$. Otherwise there exists $j$ such that 
	\[ \pi(u_j)\ne \tilde\pi. \]
	
Assume that $j>i$. 
 We have that $x=[u_1,\dots,u_{j-1}]\in O_{\tilde \pi}(N)$, and $u_j$ induces a coprime automorphism on $O_{\tilde \pi}(N)$. Therefore, $[u_1,\dots,u_j]=[x,u_j] \in \U_k  \cap N$ and it follows from Lemma \ref{lem2},  that $y=[u_1,\dots,u_k]$  is a product of $k$-boundedly many elements in $\U_k  \cap N$, which are all conjugates of $[u_1,\dots,u_j]$ or $[u_1,\dots,u_j]^{-1}$. 
	
So we can assume that $j<i$.  	By Lemma \ref{lem2} the element $x=[u_1,\dots,u_{i-1}]$ is the product of $k$-boundedly many conjugates of $u_j$ or $u_j^{-1}$, so that $[x,u_i]^{-1}=[u_i,x]$ is the product of $k$-boundedly elements of the form 
$[a,b]$, where $a\in O_{\tilde \pi}(N)$ and $b$ is a conjugate of $u_j$, so it acts
coprimely on $O_{\tilde \pi}(N)$. All elements $[a,b]$ belong to $\U_k\cap N$, thus $[x,u_i]^{-1}=[u_1,\dots,u_i]^{-1}$  is the product of $k$-boundedly many elements  in $\U_k\cap N $. Since $\U_k$ is symmetric, the result follows. 
\end{proof}

\begin{lemma}\label{conj} 
Let  $G$ be a metapronilpotent group. 
 Then  every ${\gamma_k}$-value in $G$ is a product of $k$-boundedly many elements in $\U_k$.
\end{lemma}
\begin{proof} 
Write $G =N T$, where $T$ is a system normalizer and $N= \gamma_\infty (G)$. Since $G$ is metapronilpotent, it follows that $N$ and $T$ are pronilpotent. 
 Let $x$ be a $\gamma_k$-value of $G$. Then we can write $x$ in the form
\[x=[n_1t_1,\dots,n_kt_k],\]
where $n_i\in N$ and $t_i\in T$ for all $i=1,\dots,k$.
Using the basic commutators identities $[ab,c]=[a,c]^b[b,c]$, $[a,bc]=[a,c][a,b]^c$ and $ab=b^{a^{-1}}a=b a^{b}$ we can write $x$ as
\[x=[t_1,\dots,t_k]\prod_{i=1}^{2^k-1}[x_{i1},\dots,x_{ik}]\]
where all $x_{ij}$ are conjugates of elements in $N\cup T$ and at least one entry in each $\gamma_k$-value $[x_{i1},\dots,x_{ik}]$ belongs to $N$.
As $T$ is pronilpotent, $[t_1,\dots,t_k]\in \U_k$. Moreover,
by   Lemma \ref{idea},  all $[x_{i1},\dots,x_{ik}]$  are products of $k$-boundedly many elements in $\U_k\cap N$. This concludes the proof.
  \end{proof}  

Any finite group $H$ has a normal series each of whose factors either is soluble or is a direct product of nonabelian simple groups. The nonsoluble length of $H$, denoted by $\lambda(H)$, is defined as the minimum number of nonsoluble factors in a series of this kind. It is easy to see that the nonsoluble length $\lambda(H)$  is equal to the least positive integer $l$ such that there is a series of characteristic subgroups
$$1=L_0\le R_0 < L_1 \le R_1 < \dots\le R_{l}=H$$
 in which each quotient $L_i/R_{i-1}$ is a (nontrivial) direct product of nonabelian simple groups, and each quotient $R_i/L_i$ is soluble (possibly trivial).

It is natural to say that a profinite group $G$ has finite nonprosoluble length at most $l$ if $G$ has a normal series 
 $$1=L_0\le R_0 < L_1 \le R_1 < \dots\le R_{l}=G$$
in which each quotient  $L_i/R_{i-1}$ is a (nontrivial) Cartesian product of nonabelian
finite simple groups, and each quotient $R_i/L_i$ is prosoluble (possibly trivial). 
 In particular if, for some positive integer $m,$ all continuous
finite quotients of a profinite group $G$ have nonsoluble length at most $m,$ then $G$ 
has finite nonprosoluble length at most $m$ (see e.g.  \cite[Lemma 2]{wilson-compact}).

In the rest 
 of this section $G$ will be a  profinite group such that $\U_k\subseteq \Delta (G)$. Of course, in every quotient $G/N$ of $G$ the elements of $\U_k(G/N)$ are $FC$-elements. 
\begin{lemma}\label{simple} Suppose $G$ is a direct product of nonabelian finite simple groups. Then $G$ is finite.
\end{lemma}

\begin{proof} Let $G=\Pi_{i\in I} S_i$, where each $S_i$ is a nonabelian finite simple group. In every factor $S_i$ choose a nontrivial element $a_i\in \U_{k}(S_i)$.  Note that $a=\prod_{i\in I}a_i\in \U_k$. Further, observe that $C_G(a)=\prod_{i\in I}C_{S_i}(a_i)$ has finite index in $G$ if and only if $I$ is finite. This proves the result. 
\end{proof}

\begin{lemma}\label{prosol} The group $G$ is virtually prosoluble.
\end{lemma}
\begin{proof} Suppose that $G$ is not prosoluble. Let $P$ be a Sylow $2$-subgroup of $G$. Since $\U_k(P)=P_{\gamma_k}$, it follows from Proposition \ref{maingamma} that $P$ is virtually nilpotent. Thus $P$ is soluble, say of derived length $d$. By \cite[Theorem 1.4]{KS2} every finite image of $G$ has nonsoluble length at most $d$ and so also $G$ has nonprosoluble length at most $d$.

Let 
$$1=L_0\le R_0 < L_1 \le R_1 < \dots\le R_{s}=G$$ 
be a normal series of finite length in which each section $L_i/R_{i-1}$ is a (nontrivial) direct product of nonabelian finite simple groups, and each section $R_i/L_i$ is prosoluble (possibly trivial). By Lemma \ref{simple}, every section $L_i/R_{i-1}$ is finite. Let $$H=\{x\in G \mid \ [L_i,x]\leq R_{i-1}\text{ for all } i\}$$ be the centralizer in $G$ of the non-prosoluble sections of the series. This is an open prosoluble subgroup of $G$.
\end{proof}

\begin{lemma}\label{F(G)} If $G$ is prosoluble, then $F(G)\ne 1$.
\end{lemma}
\begin{proof}
Assume by contradiction that $F(G)=1$ and let $x$ be a nontrivial element in $\U_k$. Then $K=\langle x^G\rangle$ is generated by finitely many conjugates of $x$ and so  its centralizer has finite index in $G$. It follows that the centre $Z(K)$ of $K$ has finite index in $K$. Moreover $Z(K)\le F(G)=1$. Therefore $Z(K)=1$ and consequently $K$ is finite. As $G$ is prosoluble, $K$ is soluble and $F(K)\ne 1$. As $F(K)\le F(G)$,  this contradicts the assumption that $F(G)=1$.
\end{proof}

Now we are ready to prove Proposition \ref{mainFC}.
\begin{proof}[Proof of Proposition \ref{mainFC}] 
By Lemma \ref{prosol} we may assume that $G$ is prosoluble and so  $F(G)\ne 1$ by Lemma \ref{F(G)}. If $F(G)=G$, then all $\gamma_k$-values of $G$ are contained in $\U_k$ and the result is immediate from Proposition \ref{maingamma}.

Assume that $G$ is not pronilpotent and suppose first that $G=F_2(G)$. Then $G=NT$ with  $N=\gamma_{\infty}(G)$ and $T$ a system normalizer. In this situation $N$ and $T$ are pronilpotent and in view of Lemma \ref{conj} every $\gamma_k$-value in $G$ is a product of finitely many elements from $\U_k$. Thus all $\gamma_k$-values belong to the $FC$-centre $\F$ of $G$ and it follows from Proposition \ref{maingamma} that $G$ is virtually nilpotent. 
 
This proves that $F_2(G)$ is virtually nilpotent, whence  $F_2(G)/F(G)$ is finite. Let $\bar G=G/F(G)$, and let $\bar K$ be an open normal subgroup of $\bar G$ such that $\bar K\cap F(\bar G)=1$. Then $F(\bar K)=1$ which, because of Lemma \ref{F(G)}, implies that $\bar K=1$. Hence  $\bar G$ is finite and $G$  is virtually nilpotent. This concludes the proof. 
\end{proof}

\section{Proof of Theorem 1}

 We start with the case where $G$ is a metapronilpotent group. 
\begin{lemma}\label{General case} 
Let  $G$ be a metapronilpotent group in which the centralizers of {\rm u}$_k$-commutators are either finite or open. 
Then $G$ is virtually nilpotent.
\end{lemma} 
\begin{proof}
Write $G =N T$, where $T$ is a system normalizer and $N= \gamma_\infty (G)$. Since $G$ is metapronilpotent, it follows that $N$ and $T$ are pronilpotent. Hence, by Proposition \ref{maingamma} both $N$ and $T$ are  virtually nilpotent. Let  $N_0\leq N$ and  $T_0\leq T$ be nilpotent open subgroups of $N$ and $T$ respectively. Remark that  $N_0$ can be chosen characteristic in $G$. 
Notice that $L=N_0T_0$ is open in $G$ and it is sufficient to show that $L$ is virtually nilpotent. Obviously, $L$ is soluble. Arguing by induction on the derived length of $L$ we can assume that $L'$ is virtually nilpotent. Then $L$ has an open subgroup whose commutator subgroup is nilpotent and so without loss of generality we can assume that $L'$ is nilpotent.

 If $L'$ is finite, then $L$ is  virtually abelian. Assume that $L'$ is  infinite. Lemma \ref{center} says that $Z(L')$ is infinite, too.  Observe that every element of $\U_k(L)$ centralizes $Z(L')$. Hence, every element of $\U_k(L)$ is an $FC$-element, and therefore $L$ is virtually nilpotent by Proposition \ref{mainFC}.  
 \end{proof} 
Recall that a group is locally finite if every finite subset is contained in a finite subgroup.
\begin{lemma}\label{centralizer} Let $G$ be a profinite group and $H$ a locally finite normal abstract subgroup of $G$. Let $\hat H$ be the closure of $H$. Suppose that
$C_H(a) = 1$ for some torsion element $a \in G.$ Then $C_{G/ \hat H} (a)$ is the image of $C_G(a)$ in $G/ \hat H.$
\end{lemma}
\begin{proof}
Let $h \in H$. Then $J=\langle h \rangle^{\langle a \rangle}$ is a finite $a$-invariant subgroup of $H$. Consider the map $f : J \mapsto J$ defined by: $x\mapsto [x,a]$. Since $C_H(a) = 1$, the map is injective and hence surjective.  So $h = [x,a]$ for some $x \in J$. This holds for every $h\in H$ and therefore the natural extension of $f$ to 
 $\hat H$ is also surjective.

Clearly,  $C_{G/ \hat H} (a)$ contains the image of $C_G(a)$ in $G/\hat H$. Conversely, if $x\hat H \in  C_{G/ \hat H} (a)$, then $[x,a] \in \hat H$, whence
$[x,a]=[h,a]$ for some $h \in \hat H$. We deduce that $x h^{-1} \in C_G(a)$ and $x \in C_G(a)\hat H$. Thus $C_{G/ \hat H} (a)=C_G (a)\hat H /\hat H $, as required.
\end{proof}

Now we are ready to prove the  result on virtually nilpotency.

\begin{proposition}\label{mainfirst} Let  $G$ be a profinite group in which the centralizers of uniform $k$-step commutators  are either finite or open. Then $G$ is virtually nilpotent.
\end{proposition}

\begin{proof}
Suppose that an element $x \in \U_k$ has infinite order. Then the centralizer $C=C_G(x)$ is infinite and therefore open. Every {\rm u}$_k$-commutator lying in $C$ is centralized by $x$, which has infinite order. Therefore every {\rm u}$_k$-commutator in $C$ is an $FC$-element. Apply Proposition \ref{mainFC} to deduce that $C$ is  virtually nilpotent, and the same holds for $G$,  as $C$ is open in $G$.  Hence, without loss of generality, we will assume that every element of $\U_k$ has finite order.
 
Let $Y\subseteq \U_k$ be the set of elements of $\U_k$ that have open centralizers, and let $H$ be the abstract subgroup generated by $Y$. Note that $H$ is contained in the $FC$-centre  of $G$. In particular, $H$ is locally finite. If $Y=\U_k$, the result is immediate from Proposition \ref{mainFC} so assume that $Y\neq \U_k$.

 Let $a \in \U_k \setminus Y$. As the centralizer $C_G(a)$ is finite and $H$ is residually finite,  there exists a normal subgroup $K$ of finite index in $H$ such that $C_K(a)=1$. Let $\hat H$ and $\hat K$ be the topological closures of $H$ and $K$, respectively. It follows from Lemma \ref{centralizer} that $C_{G/ \hat K} (a)=C_G (a)\hat K /\hat K $ and therefore $C_{G/ \hat K} (a)$ is finite.  Since $\hat H/\hat K$ is finite, also $C_{G/ \hat H} (a)$ is finite (see e.g. \cite[Lemma 2.1]{shalev}).
 Thus every nontrivial element of $\U_k$ has finite centralizer in $G/ \hat H$. In view of Proposition \ref{mainFinite} we conclude that $G/ \hat H$ is virtually nilpotent.

 Let us now examine the action of $a$ on $H$. Again, $K$ is a finite index $a$-invariant subgroup of $H$ such that $C_K(a)=1$. A well-known corollary of the classification of finite simple groups is that a finite group admitting a fixed-point-free automorphism is soluble (see e.g.  \cite{rowley}). Thus $K$ is locally soluble. 
  Recall that  a Carter subgroup is a self-normalizing nilpotent subgroup and note 
 that $\langle a \rangle$ is a Carter subgroup in every finite subgroup $T$ of $K\langle a \rangle$  such that $a\in T$. 
 The main result in Dade's paper \cite{Dade}  implies that the Fitting height $h(T)$ of $T$ is bounded by a function depending only on the order of $a$. 
  We deduce that $K$ has a characteristic series of finite length all of whose factors are locally nilpotent. Therefore $\hat K$ has a  finite characteristic series with pronilpotent factors, that is, $\hat K$ has finite Fitting height. As $\hat K$ is open in $\hat H$ and $G/ \hat H$ is virtually nilpotent, we conclude that also $G$ has an open prosoluble subgroup whose Fitting height is finite. So, without loss of generality we can assume that $G$ is prosoluble and $h(G)$ is finite. 
 
If $h(G) \le 2$, Lemma \ref{General case} implies that $G$ is virtually nilpotent. Assume that $h(G) > 2$ and argue by induction of $h(G)$. It follows that $G/F(G)$ is virtually nilpotent. Hence, $G$ has an open subgroup $M$ such that $h(M)\leq2$. In view of Lemma \ref{General case} $M$ (and therefore $G$) is virtually nilpotent. The proof is now complete.
 \end{proof}

Now Theorem \ref{main2}  follows. 

\begin{proof}[Proof of Theorem \ref{main2}] Let  $G$ be a profinite group in which the centralizers of elements of $\U_2$ are either finite or open. It follows from Proposition \ref{mainfirst}  that  $G$ is virtually nilpotent.
  So   we can assume that $G$ is nilpotent. In that case every  commutator in $G$ is a uniform commutator. In view of Proposition \ref{maingamma} we conclude that $G$ has an open subgroup which is nilpotent of class at most $3$. 
  \end{proof}

It remains to prove the part of Theorem \ref{main} which states that $\gamma_k(G)$ is virtually abelian.

Let G be a group and $w=w(x_1,\dots,x_n)$   a word. The marginal subgroup $w^*(G)$ of
$G$ corresponding to the word $w$ is defined as the set of all $x \in G$  such that
 $$w(g_1,\dots, x g_i,\dots,g_n)= w(g_1,\dots,  g_i x ,\dots,g_n)=w(g_1,\dots,g_i,\dots,g_n)$$
 for all  $g_1,\dots,g_n \in G$ and $1 \le i \le n$. 
 It is well known that $w^*(G)$ is a characteristic subgroup of $G$ and  $[w^*(G), w(G)]=1$. 
   Note that marginal subgroups in profinite groups are closed. 

Let  $S$ be a subset of a group $G$. Following \cite{DMSrestricted} 
 define the  $w^*$-residual of  $S$ of $G$ to be the intersection  of all normal subgroups $N$ such that $SN/N$ is contained in the marginal subgroup  $w^*(G/N)$.

 For  multilinear commutator words the  $w^*$-residual of  a normal subgroup has the following properties.  

\begin{lemma}\label{ts}\cite[Lemma 4.1]{DMSrestricted} 
Let $w$ be a multilinear 
commutator word, $G$  a group and $N$ a normal subgroup of $G$. 
 Then the   $w^*$-residual of $N$ in $G$ is the subgroup generated by the elements $w(g_1, \dots, g_n)$ where at least one of $g_1, \dots, g_n$ belongs to $N$. \end{lemma} 
 
 \begin{lemma}\label{concise}\cite[Lemma 4.2]{DMSrestricted}
 Let $w$ be a multilinear commutator word, $G$  a profinite group and  $N$  an open normal subgroup of $G$. Then the $w^*$-residual of $N$ is open in $w(G)$. 
\end{lemma}

The following result is a particular case of  \cite[Proposition 4.5]{DMSrestricted}: 
\begin{proposition}\label{X} 
 Let $G$ be a profinite group and $N$ a normal subgroup of $G$. Let $H$ be the topological closure of $\Delta(G)$ in $G$. 
  Fix $i\in \{1,\dots,k\}$ and  consider the set 
 $X_i=\{[g_1,\dots,g_k] \mid g_i\in N,g_j\in G \}$. 
 If  
 \[ X_i \subseteq \Delta(G), \]
then $[H , \langle X_i\rangle ]$ is finite.
\end{proposition}

    Following the lines of \cite[Theorem 4.3]{DMSrestricted}, we have: 

\begin{theorem}\label{genN}
Assume that $G$ is a  profinite group  in which the centralizers of u$_k$-commutators  are either finite or open and $N$ is an infinite normal  nilpotent subgroup of $G$.
 Then the $\gamma_k^*$-residual of $N$ has finite commutator subgroup.  
\end{theorem}
\begin{proof}
 For $i=1, \dots n$, let $X_i$ be the set of $\gamma_k$-values   $[g_1,...,g_k]$ such that $g_i$ belongs to $N$. 
  It follows from Lemma \ref{idea} that every element  in  $X_i$ is a product of $k$-boundedly many elements in $\U_k\cap N$. 
    As $N$ is infinite, the center $Z(N)$ of $N$ is infinite as well, thus every element in $\U_k\cap N$ is an $FC$-element. 
     Therefore also 
  the set $X_i$  consists of $FC$-elements.
   It follows   from Proposition \ref{X} that $[H, \langle X_i \rangle]$ is finite for every $i$. 
  By Lemma \ref{ts},   the $\gamma_k^*$-residual of $N$ is the subgroup $R$ generated by the set $X= X_1 \cup \dots  \cup X_k $. 
 Thus $[H, R]= \prod_{i=1}^{k}[H,\langle X_i \rangle]$ is finite. Finally, note that $R\le H$ and so $R'\le [H,R]$ is also finite.
 \end{proof}

Now Theorem  \ref{main}  follows.

\begin{proof}	[Proof of Theorem \ref{main}]	
Assume that $G$ is a  profinite group  in which the centralizers of uniform $k$-step commutators  are either finite or open. 
 We proved in Proposition \ref{mainfirst} that $G$ is virtually nilpotent. Now we will show that $\gamma_k(G)$ is abelian-by-finite. 
Let $N$ be an open nilpotent subgroup of $G$. Of course we can assume that $N$ is infinite. By Theorem \ref{genN},  the $\gamma_k^*$-residual  $R$ of $N$ has finite commutator subgroup, thus it is virtually abelian.  Moreover, by Lemma \ref{concise} $R$ is open in $\gamma_k(G)$. Thus $\gamma_k(G)$ is virtually abelian and the proof is complete. 
\end{proof}

\section{Strong conciseness of uniform commutators}

A word $w$ is said to be concise in a class of groups $\mathcal X$ if $w(G)$ is finite whenever the set of $w$-values in $G$ is finite for a group $G\in\mathcal X$. In the sixties Hall raised the problem whether all  words are concise,
 but in 1989 S. Ivanov \cite{ivanov} solved the problem in the negative  (see also \cite[p.\ 439]{ols}). On the other hand, the problem for residually finite groups remains open (cf. Segal \cite[p.\ 15]{Segal} or Jaikin-Zapirain \cite{jaikin}). In recent years several  positive results with respect to this problem were obtained (see \cite{AS1, gushu,  fealcobershu, dms-conciseness, dms-bounded, DMS-engel-II}). 

A word $w$ is called boundedly concise in a class of groups $\mathcal X$ if whenever the set of its values is finite of size at most $m$ in a group $G\in \mathcal X$, it always follows that the
subgroup $w(G)$ is finite of order bounded by a function of $m$ and $w$.
  In \cite{fernandez-morigi} it is shown  that every word which is concise in the class of all groups is actually boundedly concise. 
 There is a conjecture that every word which is concise in residually finite groups is boundedly concise (cf. \cite{fealcobershu}) but this probably  will remain open for some time. 

 On the other hand, the multilinear commutator words and words implying virtual nilpotency are known to have this property \cite{dms-conciseness}. Recall that a word $w$ is said to imply virtual nilpotency if every finitely generated metabelian group $G$ where $w$ is a law 
  has a nilpotent subgroup of finite index. 
 It follows from Gruenberg's result \cite{gruen} that the Engel words imply virtual nilpotency, so they are boundedly concise. 

The main result in \cite{fernandez-morigi} states that if $w$ is a multilinear commutator word and the set of $w$-values in a group $G$ has size at most $m$, then  the verbal subgroup $w(G)$  is finite of order bounded by a function of $m$, independently of $w$. We will show that in the class of profinite groups the set $\U_k$ has the same property.

Recall that the $i$-th centre $Z_i(G)$ of a  group $G$ is defined inductively by 
 $Z_0(G)=1$  and $Z_i(G)/Z_{i-1}(G)=Z(G/Z_{i-1}(G))$ for $i\ge 1$.  The last term of the upper central series of a finite group $G$ will be denoted by  $Z_\infty(G)$. A classical result, due to Baer, states that if $Z_{i}(G)$ has finite index $t$ in $G$, then $\gamma_{i+1}(G)$ is finite, and its order is bounded by a function of $i$ and $t$ (see the proof of \cite[14.5.1]{rob}). Similarly, if $G$ is a finite group such that $[G:Z_\infty(G)]=t$, then the order of $\gamma_\infty(G)$ is $t$-bounded (see \cite{kos}).

\begin{proposition}\label{boundedly} 
 Let $G$ be a profinite group in which $|\U_k(G)|\le m$ for some positive integer $m$. Then  $|\gamma_k(G)|$ is $m$-bounded.
\end{proposition}
\begin{proof}  As $\gamma_k(G)$ is generated by the set $\U_k(G)$ on which $G$ acts by conjugation, it follows that $[G:C_G(\gamma_k(G)]\le m!$. Thus $Z(\gamma_k(G))$ has $m$-bounded index in $\gamma_k(G)$ and, by Schur Theorem, $\gamma_k(G)'$ has $m$-bounded order (see \cite[4.12]{robinson2}).          We may pass to the quotient $G/\gamma_k(G)'$ and, without loss of generality, assume that $\gamma_k(G)$ is abelian. Now we will prove that $|\gamma_k(G)|$ is finite and $m$-bounded by induction on $m$, the case $m=1$ being trivial. If $G$ is pronilpotent, then every $\gamma_k$-value is contained in $\U_k(G)$, so we can conclude by \cite[Theorem A]{fernandez-morigi}.
Suppose that $G$ is not pronilpotent. Then there exists a Sylow $p$-subgroup $P$ of $\gamma_k(G)$, for some prime $p$,  and a $p'$-element $a$ such that $[P,a]\ne 1$. 
 Since $P$ is abelian $[zy,a]=[y,a][z,a]$ for each $y,z\in P$, thus $[P,a]=\{[y,a]\mid y\in P\}$. It follows from Lemma \ref{bul2} that $[P,a]\subseteq \U_k(G)$ and  so it  has order at most $m$. Choose a nontrivial element $x\in [P,a]$. Then $x$ has order at most $m$ and $[G:C_G(x)]\le m$. Hence $|\langle x^G\rangle|$ has order at most $m^m$. Now we can pass to the quotient $G/\langle x^G\rangle$ and conclude by induction on $m$.
\end{proof}

As mentioned in the introduction 
 Theorem \ref{main} enables us to prove that if $G$ is a profinite group such that the cardinality of the set $\U_k(G)$ is less than $2^{\aleph_0}$, then $\gamma_k(G)$ is finite. 
Our proof relies on the fact that multilinear commutator words are strongly concise  \cite{dks}.  

The next lemma will be useful.

\begin{lemma}\cite[Lemma 2.2]{dks}\label{conjfinite} Let $G$ be a profinite group and let $x\in G$. If the conjugacy
class $x^G$ contains less than $2^{\aleph_0}$ elements, then it is finite.\end{lemma}

\begin{proof}[Proof of Theorem \ref{strong}] 
It is enough to prove that if $G$ is a profinite group such that the cardinality of the set of {\rm u}$_k$-commutators in $G$ is less than $2^{\aleph_0}$ then $\gamma_k(G)$ is finite. Under this assumption,  the conjugacy class $x^G$  of every {\rm u}$_k$-commutator $x$ is finite, by Lemma \ref{conjfinite}. Thus all {\rm u}$_k$-commutators are $FC$-elements and  $G$ is virtually nilpotent,  by Proposition \ref{mainFC}. 
Let $N$ be an open nilpotent normal subgroup of $G$. 
If $g\in G$, Lemma \ref{conj} implies that the  cardinality of the set of $\gamma_k$-values  in $N\langle g\rangle$ is less than $2^{\aleph_0}$. 
Taking into account that  multilinear commutator words are strongly concise  we conclude that $\gamma_k(N\langle g\rangle)$ is finite.  
  Choose a transversal $g_1,\dots,g_s$ of $N$ in $G$. 
As each $\gamma_k(N\langle g_i\rangle)$ is normalized by $N$, its normal closure $N_i$ is finite. Thus we can pass to the quotient over the finite normal subgroup $N_1\cdots N_s$ and assume that $\gamma_k(N\langle g_i\rangle)=1$ for $i=1,\dots,s$. Now, if $x\in N$,  the commutator $[x,{}_{k-1} g]$  is trivial  for each $g\in G$, that is, $x$ is a right Engel element. Therefore for every finite quotient $\bar G$  of $G$ the image $\bar N$ of $N$ is contained in $Z_{\infty}(\bar G)$ (see for instance \cite[12.3.7]{rob}). It follows that $\gamma_\infty(\bar G)$ has $s$-bounded order. As this happens for every finite quotient  of $G$,  $\gamma_\infty(G)$  is finite. Without loss of generality, we can pass to the quotient over  $\gamma_{\infty}(G)$ and assume that $G$ is pronilpotent. In this case, every $\gamma_k$-value is a {\rm u}$_k$-commutator. Since  multilinear commutator words are strongly concise, it follows  
  that $\gamma_k(G)$ is finite. This concludes the proof.
\end{proof}


\begin{thebibliography}{99}
\bibitem{AS1} C.~Acciarri, P.~Shumyatsky,  \textit{On words that are concise in residually finite groups}, {J. Pure Appl. Algebra}~\textbf{218} (2014), 130--134. 
\bibitem{as}  C.~Acciarri,  P.~Shumyatsky,  \textit{Profinite groups with restricted centralizers of $\pi$-elements}, Math. Z. \textbf{301} (2022), 1039--1045. 
\bibitem{joao}  J.~Azevedo, P.~Shumyatsky,  \textit{On finiteness of some verbal subgroups in profinite groups}, J. Algebra \textbf{574} (2021), 573--583. 
 \bibitem{Dade} E.~C.~Dade, \textit{Carter subgroups and Fitting heights of finite solvable groups}, Illinois J. Math. \textbf{13} (1969), 449--514. 
 \bibitem{dks} E.~Detomi, B.~Klopsch, P.~Shumyatsky, \textit{Strong conciseness in profinite groups}, J. Lond. Math. Soc. \textbf{102} (2020), 977--993.
\bibitem{dms2} E.~Detomi, M.~Morigi, P.~Shumyatsky, \textit{Centralizers of commutators in finite groups},  J. Algebra  \textbf{612} (2022), 475--486. 
\bibitem{dms-conciseness}  E.~Detomi, M.~Morigi, P.~Shumyatsky,  \textit{Words of Engel type are concise in residually finite groups},
{ Bull. Math. Sci.} \textbf{9} (2019), 1950012 (19 pages).
\bibitem{dms-bounded} E.~Detomi, M.~Morigi, P.~Shumyatsky, \textit{On bounded conciseness of Engel-like words in residually finite groups}, 
{J. Algebra} \textbf{521} (2019), 1--15. 
 \bibitem{DMS-engel-II} E.~Detomi, M.~Morigi and P.~Shumyatsky, \textit{Words of Engel type are concise in residually finite groups. Part II},  {Groups Geom. Dyn.},  \textbf{14} (2020), 991--1005.
\bibitem{DMSrestricted} E. Detomi, M. Morigi, P. Shumyatsky, \textit{Profinite groups with restricted centralizers of commutators}, 
Proc. Roy. Soc. Edinburgh Sect. A \textbf{150} (2020), 2301--2321.
\bibitem{dms} E. Detomi, M. Morigi, P. Shumyatsky,  \textit{Strong conciseness of coprime and anti-coprime commutators}, Ann. Mat. Pura Appl. (4) \textbf{200} (2021), 945--952.
\bibitem{ES} S. Eberhard, P. Shumyatsky, \textit{Probabilistically nilpotent groups of class two}, 	arXiv:2108.02021 [math.GR]
\bibitem{fernandez-morigi} 
G.\,A. Fern\'andez-Alcober and M. Morigi, \textit{Outer commutator words are uniformly concise}. {J. London Math. Soc.} \textbf{82} (2010), 581--595.

\bibitem{fealcobershu} G.\,A.   Fern\'andez-Alcober and P.   Shumyatsky,
 \textit{On bounded conciseness of words in residually finite groups}.  {J. Algebra}~\textbf{500}   (2018), 19--29.


\bibitem{gilotti-ribes-serena}  A. L. Gilotti, L. Ribes, L. Serena, 
\textit{Fusion in profinite groups}, Ann. Mat. Pura Appl. (4) 177  (1999),  349--362.

\bibitem{gruen} K. W. Gruenberg, \textit{Two theorems on Engel groups,} Proc. Cambridge Philos. Soc. \textbf{49} (1953), 377--380. 

\bibitem{gushu} R.   Guralnick and P.   Shumyatsky, \textit{On rational and concise word}, {J. Algebra}~\textbf{429} (2015), 213--217.

\bibitem{hi} G. Higman, \textit{Groups and rings having automorphisms without non-trivial fixed elements}, J. London Math. Soc. \textbf{32} (1957), 321--334. 

\bibitem{ivanov} S.~V.~Ivanov, \textit{P.   Hall's conjecture on the finiteness of verbal subgroups}. {Izv. Vyssh. Ucheb. Zaved.}~\textbf{325}  (1989), 60--70.
\bibitem{jaikin} A.~Jaikin-Zapirain, \textit{On the verbal width of finitely generated pro-p groups}. {Rev. Mat. Iberoam.} \textbf{168}  (2008), 393--412.

\bibitem{Ke} J.~L.~Kelley.   {\it General Topology.} Van Nostrand, Toronto - New York - London, 1955.
\bibitem{KS-strong} E.~Khukhro, P.~Shumyatsky, \textit{Strong conciseness of Engel words in profinite groups}, Accepted in {Mathematische Nachrichten},  arXiv:2108.11789 [math.GR]
\bibitem{KS} E.~Khukhro, P.~Shumyatsky, \textit{Compact groups with countable Engel sinks}, Bull. Math. Sci.  \textbf{11} (2021), 2050015.
\bibitem{KS2}  E.~Khukhro, P.~Shumyatsky, \textit{Nonsoluble and non-p-soluble length of finite groups}, Israel J. Math. \textbf{207} (2015), 507--525. 
\bibitem{kos} L.~A.~Kurdachenko, J.~Otal, I.~Ya.~Subbotin, \textit{On a generalization of Baer theorem}, Proc. Amer. Math. Soc. \textbf{141} (2013), 2597--2602. 
\bibitem{LP} L.~L\'evai, L.~Pyber, \textit{Profinite groups with many commuting pairs or involutions}, 
{Arch. Math. (Basel)} {\bf 75} (2000), 1--7. 
\bibitem{ols} A.\,Yu.  Ol'shanskii, \textit{Geometry of Defining Relations in Groups}, {Mathematics and its applications}  \textbf{70} (Soviet Series), Kluwer Academic Publishers, Dordrecht, 1991.

\bibitem{reid} C. D. Reid, \textit{Local Sylow theory of totally disconnected, locally compact groups},  {J. Group Theory} 
\textbf{16} (2013), 535--555. 
\bibitem{rib-zal}   L.\ Ribes,   P.\ Zalesskii, \textit{Profinite groups}, A Series of Modern Surveys in Mathematics,
40. Springer-Verlag, Berlin, 2000. 
\bibitem{rob} D.~J.~S.~Robinson,  \textit{A course in the theory of groups}, Second edition. Graduate Texts in 
Mathematics, 80. Springer-Verlag, New York, 1996.
\bibitem{robinson2}  D. J. S. Robinson, \textit{Finiteness conditions and generalized soluble groups},
Part 1. Springer-Verlag, New York-Berlin, 1972.
\bibitem{rowley}  P. Rowley, \textit{Finite groups admitting a fixed-point-free automorphism group}, J. Algebra \textbf{174} (1995), 724--727.

\bibitem{Segal} D. Segal, \textit{Words: notes on verbal width in groups}. LMS Lecture Notes \textbf{361}, Cambridge Univ. Press, Cambridge, 2009.


\bibitem{shalev}  A. Shalev, \textit{Profinite groups with restricted centralizers}, Proc. Amer. Math. Soc. \textbf{122} (1994), 1279--1284. 
\bibitem{S-BFC} P.~Shumyatsky, \textit{Finite-by-nilpotent groups and a variation of the BFC-theorem},  Mediterr. J. Math. \textbf{19} (2022), Paper No. 202, 11 pp.
\bibitem{tho} J. Thompson,  \textit{Finite groups with fixed-point-free automorphisms of prime order}, Proc. Nat. Acad. Sci. U.S.A. \textbf{45} (1959), 578--581. 
\bibitem{wilson} J. S. Wilson, {\it Profinite groups}, Clarendon Press, Oxford, 1998.
 \bibitem{wilson-compact} J. S. Wilson, \textit{On the structure of compact torsion groups}. Monatsh. Math. \textbf{96} (1983)
57--66.
\end{thebibliography}

\end{document}